\documentclass[10pt,leqno]{article}
\usepackage{hyperref}
\usepackage{xcolor}
\definecolor{labelkey}{rgb}{0,0.08,0.45}
\definecolor{refkey}{rgb}{0,0.6,0.0}
\definecolor{Brown}{rgb}{0.45,0.0,0.05}
\definecolor{nido}{rgb}{1.0,0.2,0.0}
\definecolor{dgreen}{rgb}{0.00,0.40,0.00}
\definecolor{dblue}{rgb}{0,0.08,0.45}
\definecolor{lime}{rgb}{0.00,0.8,0.0}
\definecolor{lblue}{rgb}{0.5,0.5,0.99}

\usepackage[doc,hhb]{optional}
\usepackage{exscale,relsize}
\usepackage{amsmath}
\usepackage{amsfonts}
\usepackage{amssymb}
\usepackage{calc}
\usepackage{theorem}
\usepackage{pifont}      
\usepackage{soul}      
\newcommand{\scal}[2]{\langle{{#1},{#2}}\rangle}

\newcommand{\RR}{\ensuremath{\mathbb R}}
\newcommand{\II}{\ensuremath{\mathbb I}}

\newcommand{\RX}{\ensuremath{\,\left]-\infty,+\infty\right]}}
\newcommand{\RXX}{\ensuremath{\,\left[-\infty,+\infty\right]}}

\newcommand{\thalb}{\ensuremath{\tfrac{1}{2}}}

\newcommand{\menge}[2]{\big\{{#1} \mid {#2}\big\}}

\newcommand{\spand}{\operatorname{span}}

\newcommand{\dom}{\ensuremath{\operatorname{dom}}}

\newcommand{\gra}{\ensuremath{\operatorname{gra}}}

\newcommand{\ran}{\ensuremath{\operatorname{ran}}}

\renewcommand{\phi}{\ensuremath{\varphi}}

\newcommand{\To}{\ensuremath{\rightrightarrows}}

\newcommand{\argmin}{\ensuremath{\operatorname{argmin}}}

\newtheorem{theorem}{Theorem}[section]
\newtheorem{lemma}[theorem]{Lemma}
\newtheorem{fact}[theorem]{Fact}
\newtheorem{corollary}[theorem]{Corollary}
\newtheorem{proposition}[theorem]{Proposition}
\newtheorem{definition}[theorem]{Definition}

\theoremstyle{plain}{\theorembodyfont{\rmfamily}
}
\theoremstyle{plain}{\theorembodyfont{\rmfamily}
}
\theoremstyle{plain}{\theorembodyfont{\rmfamily}
}
\theoremstyle{plain}{\theorembodyfont{\rmfamily}
\newtheorem{example}[theorem]{Example}}
\theoremstyle{plain}{\theorembodyfont{\rmfamily}
\newtheorem{remark}[theorem]{Remark}}
\theoremstyle{plain}{\theorembodyfont{\rmfamily}
}

\begin{document}


\title{{\sffamily On Borwein-Wiersma Decompositions of Monotone Linear
Relations}}

\author{
Heinz H.\ Bauschke\thanks{Mathematics, Irving K.\ Barber School,
UBC Okanagan, Kelowna, British Columbia V1V 1V7, Canada. E-mail:
\texttt{heinz.bauschke@ubc.ca}.},  Xianfu
Wang\thanks{Mathematics, Irving K.\ Barber School, UBC Okanagan,
Kelowna, British Columbia V1V 1V7, Canada. E-mail:
\texttt{shawn.wang@ubc.ca}.}, and Liangjin\
Yao\thanks{Mathematics, Irving K.\ Barber School, UBC Okanagan,
Kelowna, British Columbia V1V 1V7, Canada.
E-mail:  \texttt{ljinyao@interchange.ubc.ca}.}. }

\date{December 14, 2009}

\maketitle


\begin{abstract}
\noindent
%
%
Monotone operators are of basic importance in optimization as
they generalize simultaneously
subdifferential operators of convex functions
and positive semidefinite (not necessarily symmetric) matrices.
In 1970, Asplund
studied the additive decomposition of a maximal
monotone operator as the sum of a
subdifferential operator and an ``irreducible'' monotone operator.
In 2007, Borwein and Wiersma [\emph{SIAM J.\ Optim.}~18 (2007),
pp.~946--960] introduced another additive decomposition, where
the maximal monotone operator is written as the sum of
a subdifferential operator and a ``skew'' monotone operator.
Both decompositions are variants of the well-known
additive decomposition of a matrix via its symmetric and skew part.

This paper presents a detailed study of the Borwein-Wiersma decomposition
of a maximal monotone linear relation.
We give sufficient conditions and characterizations for a maximal monotone linear relation
 to be Borwein-Wiersma decomposable,
 and show that Borwein-Wiersma decomposability
  implies Asplund decomposability.
  We exhibit irreducible linear maximal monotone
 operators without full domain, thus
 answering one of the questions raised by Borwein and Wiersma.
The Borwein-Wiersma decomposition of any maximal monotone
linear relation is made quite explicit in Hilbert space.

\end{abstract}



\section{Introduction}

Monotone operators play important roles in convex analysis and
optimization \cite{Rocky,Si2,RockWets,ph,Zalinescu,Si,BurachikIusem,BorVan}.
In the current
literature, there are two decompositions for maximal monotone
operators: the first was
introduced by Asplund in 1970 \cite{Asplund} and the second by
Borwein and Wiersma in 2007 \cite{BorweinWiersma}.
These decompositions
express a maximal monotone operator as the sum of the subdifferential
operator of
a convex function and a singular part (either irreducible or skew),
and they can be viewed as
analogues of the well known decomposition of a matrix into the sum of a
symmetric and a skew part.
They provide intrinsic insight into
the structure of monotone operators and they have the potential to be
employed in numerical algorithms (such as proximal point algorithms
\cite{BK,rockprox}).
It is instructive to study these
decompositions for monotone linear relations
to test the general theory and include
linear monotone operators as interesting special cases
\cite{PheSim,BauBorPJM}. Our goal in this paper is to study
the Borwein-Wiersma decomposition of a maximal monotone linear relation.
It turns out that a complete and elegant
characterization of Borwein-Wiersma decomposability
exists and that the Borwein-Wiersma decomposition can be made
quite explicit (see Theorem~\ref{theo:9} and Example~\ref{EBor:1}).

The paper is organized as follows.
After presenting auxiliary results in Sections~\ref{aux1},
we show in Section~\ref{decompositionpart} that Borwein-Wiersma
decomposability always implies Asplund decomposability,
and we present some
sufficient conditions for a maximal monotone linear relation to be
Borwein-Wiersma decomposable. Section~\ref{uniqueness} is devoted to
the uniqueness of the Borwein-Wiersma decomposition, and we
characterize those linear relations that are subdifferential operators of
proper lower semicontinuous convex functions.
In Section~\ref{characterization}, it is shown that a maximal monotone
linear relation $A$ is Borwein-Wiersma decomposable if and only if
the domain of $A$ is a subset of the domain of its adjoint $A^*$.
This is followed by examples illustrating neither $A$ nor $A^*$ may be
Borwein-Wiersma decomposable. Moreover, it can happen that
$A$ is Borwein-Wiersma decomposable, whereas $A^*$ is not.
Residing in a Hilbert space either $\ell^2$ or $L^{2}[0,1]$,
our examples are irreducible linear maximal monotone operators
without full domain, and they are utilized to provide an
answer to Borwein and Wiersma's
\cite[Question~(4) in Section~7]{BorweinWiersma}.
In Section~\ref{explicit}, we
give more explicit Borwein-Wiersma decompositions in Hilbert spaces.
The paper is concluded by a summary in Section~\ref{summary}.

We start with some
definitions and terminology. Throughout this paper, we assume that
\begin{equation*}
\text{$X$ is a reflexive real Banach space, with topological dual space
$X^*$, and pairing $\scal{\cdot}{\cdot}$.}
\end{equation*}
Let $A$ be a set-valued operator from $X$ to $X^*$.
Then $A$ is \emph{monotone} if
\begin{equation*}
\big(\forall (x,x^*)\in \gra A\big)\big(\forall (y,y^*)\in\gra
A\big) \quad \scal{x-y}{x^*-y^*}\geq 0,
\end{equation*}
where
$\gra A := \menge{(x,x^*)\in X\times X^*}{x^*\in Ax}$;
$A$ is said to be \emph{maximal monotone} if no proper enlargement
(in the sense of graph inclusion) of $A$ is monotone.
The \emph{inverse
operator} $A^{-1}\colon X^*\To X$ is given by $\gra A^{-1} :=
\menge{(x^*,x)\in X^*\times X}{x^*\in Ax}$; the \emph{domain} of $A$ is
$\dom A := \menge{x\in X}{Ax\neq\varnothing}$, and its \emph{range} is $\ran A: = A(X)$.
Note that $A$ is said to be  a \emph{linear relation}
if $\gra A$ is a linear subspace of
$X\times X^*$ (see \cite{Cross}). We say
$A$ is a \emph{maximal monotone linear relation} if $A$ is a maximal
monotone operator and $\gra A$ is a linear subspace of $X\times
X^*$.
The \emph{adjoint} of $A$, written $A^*$, is defined by
\begin{equation*}
\gra A^* :=
\menge{(x,x^*)\in X\times X^*}{(x^*,-x)\in (\gra A)^\bot},
\end{equation*}where, for any subset $S$ of a reflexive Banach space $Z$ with continuous
dual space $Z^*$,
$S^\bot := \menge{z^*\in Z^*}{z^*|_S \equiv 0}$.
Let $A$ be a linear relation from $X$ to $X^*$.
We say that $A$ is
\emph{skew} if $\langle x,x^*\rangle=0,\; \forall (x,x^*)\in\gra A$;
equivalently, if $\gra A \subseteq \gra (-A^*)$.
Furthermore,
$A$ is \emph{symmetric} if $\gra A
\subseteq\gra A^*$; equivalently, if $\scal{x}{y^*}=\scal{y}{x^*}$, 
$\forall (x,x^*),(y,y^*)\in\gra A$.
By saying $A:X\To X^*$ \emph{at most single-valued},
we mean that for every $x\in X$, $Ax$ is either a singleton or empty.
In this case, we follow a slight but common abuse of notation and write
$A \colon \dom A\to X^*$. Conversely, if $T\colon D\to X^*$, we may
identify $T$ with $A:X\To X^*$, where $A$ is at most single-valued with
$\dom A = D$.
We define the \emph{symmetric part} and the \emph{skew part} of $A$ via
\begin{equation}
\label{Fee:1}
A_+ := \thalb A + \thalb A^* \quad\text{and}\quad
A_{\mathlarger{\circ}} := \thalb A - \thalb A^*,
\end{equation}
respectively. It is easy to check that $A_+$ is symmetric and
that $A_{\mathlarger{\circ}}$ is skew.

Let $x\in X$ and $C^*\subseteq X^*$.
We write $\langle x, C^*\rangle :=
\{\langle x, c^*\rangle \mid c^*\in C^*\}$.
If $\langle x, C^*\rangle=\{a\}$ for some constant $a\in\RR$, then we write
$\langle x, C^*\rangle=a$ for convenience.
For a monotone linear relation $A\colon X \To X^*$ it will be very
useful
to define the extended-valued quadratic function
(which is actually a special case of
\emph{Fitzpatrick's last function} \cite{BorVan} for the
linear relation $A$)
\begin{equation} \label{e:diequad}
q_A \colon x\mapsto\begin{cases} \thalb \scal{x}{Ax},&\text{if}\; x\in\dom A;\\
+\infty,&\text{otherwise}.\end{cases}
\end{equation}
When $A$ is linear and single-valued with full domain,
we shall use the well known fact (see, e.g., \cite{PheSim}) that
\begin{equation} \label{e:gradq}
\nabla q_A = A_+.
\end{equation}
For $f\colon X\to \RX$, set
$\dom f:= \{x\in X \mid f(x)<+\infty\}$ and let
$f^*\colon X^*\to\RXX\colon x^*\mapsto
\sup_{x\in X}(\scal{x}{x^*}-f(x))$ be
the \emph{Fenchel conjugate} of $f$. We denote by $\overline{f}$
 the \emph{lower semicontinuous} hull of $f$.
Recall that $f$ is said to be
  \emph{proper} if $\dom f\neq\varnothing$. If $f$ is convex,
   $\partial f\colon X\rightrightarrows X^*\colon x\mapsto \menge{x^*\in X^*}{(\forall y\in
X)\; \scal{y-x}{x^*} + f(x)\leq f(y)}$
is the
 \emph{subdifferential operator} of $f$. For a subset $C$ of $X$, $\overline{C}$ stands for the closure
 of $C$ in $X$.
 Write $\iota_{C}$ for the \emph{indicator function}
 of $C$, i.e., $\iota_{C}(x)=0$, if $x\in C$; and
 $\iota_C(x) = +\infty$, otherwise.
It will be convenient to work with the \emph{indicator mapping} $\II_C\colon X\to X^*$,
defined by $\II_C(x) = \{0\}$, if $x\in C$; $\II_C(x)=\varnothing$,
otherwise.

The central goal of this paper is to provide a detailed analysis
of the following notion in the context of maximal monotone linear
relations.

\begin{definition}[Borwein-Wiersma decomposition \cite{BorweinWiersma}]
The set-valued operator
$A:X\rightrightarrows X^*$ is \emph{Borwein-Wiersma decomposable} if
\begin{equation} \label{e:091206:a}
A=\partial f+S,
\end{equation}
where $f:X\to \RX$ is proper  lower semicontinuous and convex,
and where $S:X\To X^*$ is skew and at most single-valued.
The right side of \eqref{e:091206:a} is a \emph{Borwein-Wiersma
decomposition} of $A$.
\end{definition}
Note that every single-valued
linear monotone operator $A$ with full domain is
Borwein-Wiersma decomposable, with Borwein-Wiersma decomposition
\begin{equation}\label{capitalnews}
A=A_++A_{\mathlarger{\circ}}=\nabla q_A+A_{\mathlarger{\circ}}.
\end{equation}

\begin{definition}[Asplund irreducibility \cite{Asplund}]
The set-valued operator $A\colon X\To X^*$ is
\emph{irreducible} (sometimes termed ``acyclic''
\emph{\cite{BorweinWiersma}})
if whenever
$$ A=\partial f+S,$$
 with $f:X\rightarrow\RX$ proper lower semicontinuous and
  convex, and $S:X\To X^*$ monotone,
then necessarily $\ran (\partial f)|_{\dom A}$ is a singleton.
\end{definition}

As we shall see in Section~\ref{decompositionpart},
the following decomposition is
less restrictive.

\begin{definition}[Asplund decomposition \cite{Asplund}]
The set-valued operator
  $A\colon X\To X^*$ is \emph{Asplund decomposable}
if
\begin{equation}
\label{e:091206:b}
A=\partial f+S,
\end{equation}
where $f\colon X\to\RX $  is proper, lower semicontinuous, and convex,
and where $S$ is  irreducible.
The right side of \eqref{e:091206:b} is an \emph{Asplund decomposition} of
$A$.
\end{definition}

\section{Auxiliary results on monotone linear relations}\label{aux1}
In this section, we gather some basic properties about monotone linear relations, and conditions
for them to be maximal monotone.
These results are used frequently in the sequel.
We start with properties for general linear relations.

\begin{fact}[Cross]\label{Rea:1}
Let $A:X \rightrightarrows X^*$ be a linear relation.
Then the following hold.
\begin{enumerate}
\item \label{Th:26} $A0$ is a linear subspace of $X^*.$
\item \label{Th:28}$Ax=x^* +A0,\quad\forall x^*\in Ax.$
\item \label{Th:30}
$(\forall (\alpha,\beta)\in\RR^2\smallsetminus\{(0,0)\})$ $(\forall x,y\in\dom A)$
$A(\alpha x+\beta y)=\alpha Ax+\beta Ay$.
\item \label{Th:27} $\dom A^*= \menge{x\in X}{\scal{x}{A(\cdot)}
\text{ is single-valued and continuous on $\dom A$}}$.
\item \label{Th:29} $(A^*)^{-1}=(A^{-1})^*$.
\item \label{Sia:2b}$(\forall x\in \dom A^*)(\forall y\in\dom A)$
$\langle A^*x,y\rangle=\langle x,Ay\rangle$ is a singleton.
\item\label{Th:31} If $\gra A$ is closed, then $A^{**} = A$.
\item\label{Th:32} If $\dom A$ is closed, then $\dom A^*$ is closed.
\end{enumerate}
\end{fact}
\begin{proof}
\ref{Th:26}: See \cite[Corollary~I.2.4]{Cross}.
\ref{Th:28}: See \cite[Proposition~I.2.8(a)]{Cross}.
\ref{Th:30}: See \cite[Corollary~I.2.5]{Cross}.
\ref{Th:27}: See \cite[Proposition~III.1.2]{Cross}.
\ref{Th:29}: See \cite[Proposition~III.1.3(b)]{Cross}.
\ref{Sia:2b}: See \cite[Proposition~III.1.2]{Cross}.
\ref{Th:31}: See \cite[Exercise~VIII.1.12]{Cross}.
\ref{Th:32}: See \cite[Corollary~III.4.3(a), Proposition~III.4.9(i)(ii),
 Theorem~III.4.2(a) and Corollary~III.4.5]{Cross}.
\end{proof}

Additional information is available when dealing with \emph{monotone}
linear relations.

\begin{fact}\label{linear}
Let $A:X\To X^*$ be a monotone linear relation.
Then the following hold.
\begin{enumerate}
\item\label{Nov:s1}
$\dom A \subseteq (A0)^{\perp}$ and $A0\subseteq
(\dom A)^\bot.$
\item\label{Nov:s2}
The function
$\dom A\to\RR\colon y\mapsto \langle y, Ay\rangle$
is well defined and convex.
\item\label{Sia:2c}
For every $x\in (A0)^{\perp}$, the function
$\dom A\to\RR\colon y\mapsto \langle x, Ay\rangle$
is well defined and linear.
\item \label{sia:3iv}
If $A$ is maximal monotone, then
$\overline{\dom A^*}=\overline{\dom A}=(A0)^\bot$ and
$A0=A^*0=A_+0 = A_{\mathlarger{\circ}}0= (\dom A)^{\perp}$.
\item \label{sia:3v}
If $\dom A$ is closed, then:
$A$ is maximal monotone $\Leftrightarrow$ $(\dom A)^\bot = A0$.
\item \label{sia:3vi}
If $A$ is maximal monotone and $\dom A$ is closed,
then $\dom A^*=\dom A$.
\item \label{sia:3vii}
If $A$ is maximal monotone and $\dom A \subseteq \dom A^*$,
then
$A=A_{+}+A_{\mathlarger{\circ}}$, $A_{+}=A-A_{\mathlarger{\circ}}$
, and $A_{\mathlarger{\circ}}=A-A_{+}$.
\end{enumerate}
\end{fact}

\begin{proof}
\ref{Nov:s1}: See \cite[Proposition~2.2(i)]{BWY3}.
\ref{Nov:s2}: See \cite[Proposition~2.3]{BWY3}.
\ref{Sia:2c}: See \cite[Proposition~2.2(iii)]{BWY3}.
\ref{sia:3iv}:
By \cite[Theorem~3.2]{BWY3}, we have $A0 = A^*0 = (\dom A)^\bot = (\dom
A^*)^\bot$ and $\overline{\dom A} = \overline{\dom A^*}$.
Hence $(A0)^\bot = (\dom A)^{\bot\bot} = \overline{\dom A}$.
By Fact~\ref{Rea:1}\ref{Th:26}, $A0$ is a linear subspace of $X^*$.
Hence $A_+0 = (A0 + A^*0)/2 = (A0 + A0)/2 = A0$ and similarly
$A_{\mathlarger{\circ}}0 = A0$.
\ref{sia:3v}: See \cite[Corollary~6.6]{BWY3}.
\ref{sia:3vi}: Combine \ref{sia:3iv} with Fact~\ref{Rea:1}\ref{Th:32}.
\ref{sia:3vii}: We show only the proof of $A=A_{+}+A_{\mathlarger{\circ}}$
as the other two proofs are analogous.
Clearly, $\dom A_{+}=\dom A_{\mathlarger{\circ}}
=\dom A\cap\dom A^*=\dom A$.
Let $x\in\dom A$, and $x^*\in Ax$ and $y^*\in A^*x$. We write
$x^*=\tfrac{x^*+y^*}{2}+\tfrac{x^*-y^*}{2}\in (A_{+}+A_{\mathlarger{\circ}})x$.
Then, by Fact~\ref{Rea:1}\ref{Th:28},
$Ax=x^*+A0=x^*+(A_{+}+A_{\mathlarger{\circ}})0=
(A_{+}+A_{\mathlarger{\circ}})x$.
Therefore, $A=A_{+}+A_{\mathlarger{\circ}}$.
\end{proof}

\begin{proposition}\label{MPP}
Let $S\colon X\To X^*$ be a linear relation such
that $S$ is at most single-valued.
Then $S$ is skew
if and only if
$\langle Sx,y\rangle=-\langle Sy,x\rangle,\; \forall x, y\in\dom S$.
\end{proposition}
\begin{proof}
``$\Rightarrow$'': Let $x,y\in\dom S$. Then
$0=\langle S(x+y), x+y\rangle=\langle Sx,x\rangle+
\langle Sy,y\rangle+\langle Sx,y\rangle+\langle Sy,x\rangle
=\langle Sx,y\rangle+\langle Sy,x\rangle$.
Hence $\langle Sx,y\rangle=-\langle Sy,x\rangle$.
``$\Leftarrow$'': Indeed, for $x\in \dom S$, we have
$\scal{Sx}{x}=-\scal{Sx}{x}$ and so $\scal{Sx}{x}=0$.
\end{proof}

\begin{fact}[Br\'ezis-Browder] \emph{(See \cite[Theorem~2]{Brezis-Browder}.)}
\label{Sv:7}
Let $A\colon X \To X^*$ be a monotone linear relation
such that $\gra A$ is closed. Then the following are equivalent.
\begin{enumerate}
\item
$A$ is maximal monotone.
\item
$A^*$ is maximal monotone.
\item
$A^*$ is monotone.
\end{enumerate}
\end{fact}

\begin{fact}[Phelps-Simons]\emph{(See \cite[Corollary 2.6 and Proposition~3.2(h)]{PheSim}.)}\label{F:1}
Let $A\colon X\rightarrow X^*$ be  monotone and linear. Then $A$ is
maximal monotone and continuous.
\end{fact}

\begin{remark}Fact~\ref{F:1} also holds in locally convex spaces,
see \cite[Proposition~23]{VZ08}.
\end{remark}

\begin{proposition}\label{symm=adj}
Let $A\colon X\To X^*$ be a maximal monotone linear relation.
Then $A$ is symmetric $\Leftrightarrow$ $A=A^*$.
\end{proposition}
\begin{proof} ``$\Rightarrow$":
Assume that $A$ is symmetric, i.e., $\gra A\subseteq\gra A^*$.
Since $A$ is maximal monotone, so is $A^*$ by Fact~\ref{Sv:7}.
Therefore, $A = A^*$.
``$\Leftarrow$": Obvious.
\end{proof}

{Fact}~\ref{linear}\ref{sia:3v} provides a characterization of
maximal monotonicity for certain monotone linear relations.
More can be said in finite-dimensional spaces.
We require the following lemma, where $\dim F$ stands for the dimension
of a subspace $F$ of $X$.

\begin{lemma}\label{pm:1}
Suppose that $X$ is finite-dimensional and let
$A\colon X\To X^*$ be a linear relation.  Then
$\dim (\gra A)=\dim (\dom A)+\dim A0$.
\end{lemma}
\begin{proof}
We shall construct a basis of $\gra A$.
By Fact~\ref{Rea:1}\ref{Th:26}, $A0$ is a linear subspace.
Let $\{x^*_1,\ldots,x^*_k\}$ be a basis of $A0$, and let
$\{x_{k+1},\ldots,x_l\}$ be a basis of $\dom A$.  From
Fact~\ref{Rea:1}\ref{Th:28}, it is easy to show
$\{(0,x^*_1),\ldots,(0,x^*_k),(x_{k+1},x^*_{k+1}),\ldots,(x_l,x^*_l)\}$ is
a basis of $\gra A,$\ where $x^*_i\in Ax_i,\ i\in\{k+1,\ldots,l\}.$
Thus $\dim (\gra A)=l=\dim (\dom A)+\dim A0.$
\end{proof}

Lemma~\ref{pm:1} allows us to get a satisfactory
characterization of maximal monotonicity of
linear relations in finite-dimensional spaces.

\begin{proposition}
Suppose that $X$ is finite-dimensional, set $n=\dim X$, and
let $A\colon X\To X^*$ be a monotone linear relation.
Then $A$ is maximal monotone if and only if $\dim\gra A=n$.
\end{proposition}

\begin{proof}
Since linear subspaces of $X$ are closed,
we see from {Fact}~\ref{linear}\ref{sia:3v} that
\begin{equation}
\label{pp1}
A \text{ is   maximal  monotone} \Leftrightarrow
\dom A=(A0)^\bot.
\end{equation}
Assume first that $A$ is maximal monotone. Then $
\dom A=(A0)^\bot.$ By Lemma~\ref{pm:1},
$ \dim (\gra A)=\dim (\dom A)+\dim (A0)
=\dim ((A0)^\bot)+\dim (A0) =n$.
Conversely, let $\dim (\gra A)=n$.
By Lemma~\ref{pm:1}, we have
that $\dim (\dom A)=n-\dim (A0)$.
As $\dim ((A0)^\bot)=n-\dim (A0)$  and $\dom A\subseteq (A0)^\bot$
by Fact~\ref{linear}\ref{Nov:s1},
we have that $\dom A=(A0)^\bot$.
By \eqref{pp1}, $A$ is maximal monotone.
\end{proof}

\section{Borwein-Wiersma decompositions}
\label{decompositionpart}

The following fact, due to Censor, Iusem and Zenios \cite{censor,iusem},
was previously known in $\RR^{n}$. Here we
give a different proof and extend the result to Banach spaces.

\begin{fact}[Censor, Iusem and Zenios]\label{paramono}
The subdifferential operator
of a proper lower semicontinuous convex function $f\colon X\to \RX$ is
\emph{paramonotone}, i.e., if
\begin{equation}\label{subinclusion}
x^*\in\partial f(x), \quad y^*\in
\partial f(y),
\end{equation}
and \begin{equation}\label{equaltozero}
\scal{x^*-y^*}{x-y}=0,
\end{equation}
then $x^*\in \partial f(y)$ and $y^*\in \partial f(x)$.
\end{fact}
\begin{proof}
By \eqref{equaltozero},
\begin{equation}\label{anotherform}
\scal{x^*}{x}+\scal{y^*}{y}=\scal{x^*}{y}+\scal{y^*}{x}.
\end{equation}
By
\eqref{subinclusion},
$$f^*(x^*)+f(x)=\scal{x^*}{x}, \quad f^*(y^*)+f(y)=\scal{y^*}{y}.$$
Adding them, followed by using \eqref{anotherform}, yields
$$f^*(x^*)+f(y)+f^*(y^*)+f(x)=\scal{x^*}{y}+\scal{y^*}{x},$$
$$[f^*(x^*)+f(y)-\scal{x^*}{y}]+[f^*(y^*)+f(x)-\scal{y^*}{x}]=0.$$
Since each bracketed term is nonnegative, we must have
$f^*(x^*)+f(y)=\scal{x^*}{y}$ and $f^*(y^*)+f(x)=\scal{y^*}{x}$.
It follows that $x^*\in \partial f(y)$ and that $y^*\in \partial f(x)$.
\end{proof}

The following result provides a powerful
criterion for determining whether a given operator is irreducible and
hence Asplund decomposable.

\begin{theorem}\label{acyclicmono}
Let $A:X\To X^*$ be monotone and at most single-valued.
Suppose that there exists a dense subset $D$ of $\dom A$ such that
$$ \scal{Ax-Ay}{x-y}=0 \quad \forall x,y \in D.$$
Then $A$ is irreducible and hence Asplund decomposable.
\end{theorem}

\begin{proof}
Let $a\in D$ and $D':=D-\{a\}$.
Define $A':\dom A-\{a\}\rightarrow A(\cdot+a)$.
Then $A$ is irreducible if and only if $A'$ is irreducible.
Now we show $A'$ is irreducible.
By assumptions, $0\in D'$ and
\begin{align*}
\scal{A'x-A'y}{x-y}=0 \quad \forall x,y \in D'.\end{align*}
Let $A'=\partial f+R$, where $f$ is proper lower semicontinuous and
convex, and $R$ is monotone.
 Since $A'$ is single-valued on $\dom A'$, we have that
 $\partial f$ and $R$ are single-valued on $\dom A'$ and that
 $$R=A'-\partial f\quad \text{ on }\dom A'.$$
 By taking $x_{0}^*\in \partial f(0)$,
rewriting $A'=(\partial f-x_{0}^*)+(x_{0}^*+R)$, we can and do
suppose $\partial f(0)=\{0\}$.
 For $x,y\in D'$ we have
 $\langle A'x-A'y,x-y\rangle=0$. Then for $x,y\in D'$
 \begin{equation*}
0\leq\langle R(x)-R(y),x-y\rangle=\langle A'x-A'y, x-y\rangle-
\langle \partial f(x)-\partial f(y),x-y\rangle
=-\langle \partial f(x)-\partial f(y),x-y\rangle.
\end{equation*}
On the other hand, $\partial f$ is monotone,
thus,
\begin{align} \langle \partial f(x)-\partial f(y),x-y\rangle=0,
\quad \forall x,y\in D'.\label{Ac:1}\end{align} Using $\partial
f(0)=\{0\}$,
\begin{align}\langle\partial f(x)-0,x-0\rangle=0,\quad \forall x\in D'.\label{Ac:2}\end{align}
As $\partial f$ is paramonotone by Fact~\ref{paramono}, $\partial
f(x)=\{0\}$ so that $x\in \argmin f$. This implies that
$D'\subseteq \argmin f$ since $x\in D'$ was chosen arbitrarily.
As $f$ is lower semicontinuous, $\argmin f$ is closed.
Using that $D'$ is dense in
$\dom A'$, it follows that $\dom A'\subseteq\overline {D'}\subseteq\argmin
f$. Since $\partial f$ is single-valued on $\dom A'$, $\partial
f(x)=\{0\},\ \forall x\in \dom A'$. Hence $A'$ is irreducible,
and so is $A$.
\end{proof}

\begin{remark}
In {Theorem}~\ref{acyclicmono},
the assumption that $A$ be at most single-valued
is important: indeed, let $L$ be a proper subspace of $\RR^{n}$.
Then $\partial \iota_{L}$ is a linear
relation and skew, yet $\partial\iota_{L} = \partial\iota_{L}+0$
is not irreducible.
\end{remark}

Theorem~\ref{acyclicmono} and the definitions of the two decomposabilities
now yield the following.

\begin{corollary} Let $A:X\To X^*$ be maximal monotone such that
$A$ is Borwein-Wiersma decomposable.
Then $A$ is Asplund decomposable.
\end{corollary}

We proceed to give a few sufficient conditions
for a maximal monotone linear relation to be
Borwein-Wiersma decomposable.
The following simple observation will be needed.

\begin{lemma}\label{amy}
Let $A:X\To X^*$ be a monotone linear relation such that
$A$ is Borwein-Wiersma decomposable, say
$A=\partial f +S$,
where $f\colon X\to\RX$ is proper, lower semicontinuous, and convex,
and where $S\colon X\To X^*$ is at most single-valued and skew.
Then the following hold.
\begin{enumerate}
\item\label{domlinear}
$\partial f + \II_{\dom A}\colon x\mapsto \begin{cases}
\partial f(x), &\text{if $x\in\dom A$;}\\
\varnothing, &\text{otherwise}\end{cases}$ ~~~is a monotone linear relation.
\item\label{domainpf}
$\dom A\subseteq \dom \partial f \subseteq \dom f \subseteq (A0)^{\perp}$.
\item\label{domainpf+}
If $A$ is maximal monotone, then
$\dom A\subseteq \dom \partial f \subseteq \dom f \subseteq \overline{\dom
A}$.
\item \label{domainpf++}
If $A$ is maximal monotone and $\dom A$ is closed, then
$\dom \partial f=\dom A=\dom f$.
\end{enumerate}
\end{lemma}

\begin{proof}
\ref{domlinear}:
Indeed, on $\dom A$, we see that
$\partial f=A-S$ is the difference of two linear relations.

\ref{domainpf}:
Clearly $\dom A\subseteq \dom \partial f$.
As $S0=0$, we have $A0=\partial f(0)$.
Thus,
$\forall x^*\in A0$, $x\in X$,
$$\scal{x^*}{x}\leq f(x)-f(0).$$
Then $\sigma_{A0}(x)\leq f(x)-f(0)$, where $\sigma_{A0}$ is the  support function of $A0$.
If $x\not\in (A0)^{\perp},$ then $\sigma_{A0}(x)=+\infty$
since $A0$ is a linear subspace,
 so $f(x)=+\infty,\ \forall x\not\in (A0)^{\perp}$.
 Therefore, $\dom f\subseteq (A0)^{\perp}$.
 Altogether, \ref{domainpf} holds.

\ref{domainpf+}: Combine \ref{domainpf} with
{Fact}~\ref{linear}\ref{sia:3iv}.
\ref{domainpf++}: This is clear from \ref{domainpf+}.
\end{proof}

\begin{fact}\emph{(See \cite[Proposition~3.3]{Yao}.)} \label{f:PheSim}
Let $A:X\To X^*$ be a monotone linear relation
such that  $A$ is symmetric.
Then the following hold.
\begin{enumerate}
\item \label{f:PheSim:lsc02} $q_{A}$ is convex and $\overline{q_A}+\iota_{\dom A}=q_A$.
\item \label{f:PheSim:quad}
$\gra A \subseteq\gra \partial \overline{q_A}$.
\item
\label{f:PheSim:quad+}
If $A$ is maximal monotone, then $A=\partial \overline{q_A}$.
\end{enumerate}
\end{fact}

\begin{theorem}\label{theo:4}
Let $A\colon X \To X^*$ be a maximal monotone linear relation
such that $\dom A\subseteq\dom A^*$.
Then $A$ is Borwein-Wiersma decomposable
via
$$A=\partial\overline{q_A}+S,$$
where $S$ is an arbitrary linear single-valued
selection of $A_{\mathlarger{\circ}}$.
Moreover, $\partial\overline{q_A} =A_+$ on $\dom A$.
\end{theorem}
\begin{proof} From Fact~\ref{Sv:7}, $A^*$ is monotone,
 so $A_{+}$ is monotone. By Fact~\ref{Rea:1}\ref{Sia:2b},
$q_{A_{+}}=q_{A}$, using {Fact}~\ref{f:PheSim}\ref{f:PheSim:quad},
 $\gra A_+\subseteq\gra \partial\overline{q_{A_+}}=\gra \partial\overline{q_A}$.
Let $S:\dom A\rightarrow X^*$
 be a linear selection of $A_{\mathlarger{\circ}}$
 (the existence of which is guaranteed by a standard Zorn's lemma
 argument).
  By Fact~\ref{Rea:1}\ref{Sia:2b}, $S$ is skew.
Then, by {Fact}~\ref{linear}\ref{sia:3vii}, we have
$A=A_{+}+S\subseteq\partial
\overline{q_A}+S.$ Since $A$ is maximal monotone, $A=\partial
\overline{q_A}+S$, which is the announced
Borwein-Wiersma decomposition. Moreover, on $\dom A$, we have
$\partial \overline{q_A}= A-S=A_+$.
\end{proof}

\begin{corollary}\label{Sk:2}
Let $A:X\To X^*$ be a maximal monotone linear relation such
that $A$ is symmetric.
Then $A$ and $A^{-1}$ are Borwein-Wiersma decomposable,
with decompositions
$A=\partial\overline{q_{A}} + 0$ and $A^{-1}=\partial q^*_{A} + 0$,
respectively.
\end{corollary}
\begin{proof}
Using Proposition~\ref{symm=adj} and Fact~\ref{Rea:1}\ref{Th:29}, we obtain
$A=A^*$ and $A^{-1}=(A^*)^{-1}=(A^{-1})^{*}$.
Hence, Theorem~\ref{theo:4} applies; in fact,
$A=\partial\overline{q_{A}}$ and
$A^{-1}=\partial\overline{q_{A^{-1}}}=\partial q^*_{A}$.
\end{proof}

\begin{corollary}\label{cbr:1}
Let $A\colon X \To X^*$
be a maximal monotone linear relation such that $\dom A$ is closed,
and let $S$ be a single-valued linear selection of
$A_{\mathlarger{\circ}}$.
Then $q_A = \overline{q_A}$, $A_+ = \partial q_A$ is maximal monotone, and
$A$ and $A^*$ are Borwein-Wiersma decomposable, with decompositions
$A = A_+ + S$ and $A^*= A_+ - S$, respectively.
\end{corollary}
\begin{proof}
Fact~\ref{Rea:1}\ref{Th:31} and Fact~\ref{linear}\ref{sia:3vi}
imply that $A^{**}=A$ and that $\dom A^*=\dom A$.
By Fact~\ref{Sv:7}, $A^*$ is maximal monotone.
In view of Fact~\ref{linear}\ref{sia:3vii},
$A = A_+ + A_{\mathlarger{\circ}}$ and
$A^* = (A^*)_+ + (A^*)_{\mathlarger{\circ}} =
A_+ - A_{\mathlarger{\circ}}$.
Theorem~\ref{theo:4} yields the Borwein-Wiersma decomposition
$A = \partial \overline{q_A} + S$.
Hence $\dom A \subseteq\dom \partial \overline{q_A} \subseteq
\dom\overline{q_A} \subseteq \overline{\dom A} = \dom A$.
In turn, since $\dom A = \dom A_+$ and $q_{A} = q_{A_+}$, this implies
that $\dom A_+ = \dom\partial\overline{q_{A_+}}
= \dom\overline{q_{A_+}}$.
In view of Fact~\ref{f:PheSim}\ref{f:PheSim:lsc02}\&\ref{f:PheSim:quad},
$q_{A_+} = \overline{q_{A_+}}$ and $\gra A_+ \subseteq
\gra\partial\overline{q_{A_+}}$.
By Theorem~\ref{theo:4}, $A_+ = \partial\overline{q_A}$ on $\dom A$.
Since $\dom A = \dom A_+ = \dom \partial\overline{q_A}$ and
$q_A = q_{A_+}=\overline{q_{A_+}} = \overline{q_A}$, this implies that
$A_+ = \partial q_A = \partial \overline{q_A}$ \emph{everywhere}.
Therefore, $A_+$ is maximal monotone.
Since $A_+ = (A^*)_+$ and $-S$ is a single-valued linear section of
$(A^*)_{\mathlarger{\circ}} = -A_{\mathlarger{\circ}}$, we obtain similarly
the Borwein-Wiersma decomposition $A^* = A_+-S$.
\end{proof}

\begin{theorem}\label{Sk:1}
Let $A:X\To X^*$ be a maximal monotone linear relation
such that $A$ is skew, and let $S$ be a single-valued
linear selection of $A$.
Then $A$ is Borwein-Wiersma decomposable via
$\partial\iota_{\overline{\dom A}} + S$.
\end{theorem}
\begin{proof}
Clearly, $S$ is skew.
Fact~\ref{Rea:1}\ref{Th:28}
and {Fact}~\ref{linear}\ref{sia:3iv} imply that
$A = A0 + S = (\dom A)^\bot + S = \partial\iota_{\overline{\dom A}} + S$,
as announced.
Alternatively, by \cite[Lemma~2.2]{SV}, $\dom A\subseteq\dom A^*$ and
now apply Theorem~\ref{theo:4}.
\end{proof}

Under a mild constraint qualification,
the sum of two Borwein-Wiersma decomposable
operators is also Borwein-Wiersma decomposable and
the decomposition of the sum is the corresponding sum
of the decompositions.

\begin{proposition}[sum rule] \label{Co:r2}
Let $A_1$ and $A_2$ be maximal monotone linear relations from $X$
to $X^*$. Suppose that $A_1$ and $A_2$ are Borwein-Wiersma
decomposable via
$A_1 = \partial f_1+ S_1$ and $A_2 = \partial f_2+S_2$,
respectively.
Suppose that $\dom A_1 - \dom A_2$ is closed.
Then $A_1+A_2$ is Borwein-Wiersma decomposable via
$A_1+A_2 = \partial(f_1+f_2) + (S_1+S_2)$.
\end{proposition}
\begin{proof}
By Lemma~\ref{amy}\ref{domainpf+},
$\dom A_1 \subseteq \dom f_1 \subseteq \overline{\dom A_1}$ and
$\dom A_2 \subseteq \dom f_2 \subseteq \overline{\dom A_2}$.
Hence $\dom A_1-\dom A_2 \subseteq
\dom f_1-\dom f_2\subseteq \overline{\dom A_1}-\overline{\dom A_2}
\subseteq\overline{\dom A_1-\dom A_2}=\dom A_1-\dom A_2$.
Thus, $\dom f_1-\dom f_2 = \dom A_1-\dom A_2$ is a closed subspace of $X$.
By \cite[Theorem 18.2]{Si2},
 $\partial f_1+\partial f_2=\partial (f_1+f_2)$; furthermore,
$S_1+S_2$ is clearly skew. The result thus follows.
\end{proof}

\section{Uniqueness results}
\label{uniqueness}

The main result in this section (Theorem~\ref{uniquepart}) states
that if a maximal monotone linear relation $A$ is Borwein-Wiersma
decomposable, then the subdifferential part of its decomposition is
unique on $\dom A$.
We start by showing that subdifferential operators that are
monotone linear relations are actually symmetric, which
is a variant of a well known result from Calculus.

\begin{lemma}\label{Bor:u1}
Let $f:X\to\RX$ be proper,
 lower semicontinuous, and convex.
Suppose that the maximal monotone operator $\partial f$
is a linear relation with closed domain.
Then $ \partial f = (\partial f)^*$.
\end{lemma}
\begin{proof}
Set $A := \partial f$ and $Y := \dom f$. Since $\dom A$ is closed,
\cite[Theorem~18.6]{Si2} implies that $\dom f = Y = \dom A$.
By Fact~\ref{linear}\ref{sia:3vi}, $\dom A^* = \dom A$.
Let $x\in Y$ and consider the directional derivative
$g = f'(x;\cdot)$, i.e.,
\begin{equation*}
g\colon X\to\RXX\colon y\mapsto \lim_{t\downarrow 0}
\frac{f(x+ty)-f(x)}{t}.
\end{equation*}
By \cite[Theorem~2.1.14]{Zalinescu},
$\dom g = \bigcup_{r\geq 0}r\cdot(\dom f - x) = Y$.
On the other hand, $f$ is lower semicontinuous on $X$.
Thus, since $Y=\dom f$ is a Banach space,
$f|_Y$ is continuous by \cite[Theorem~2.2.20(b)]{Zalinescu}.
Altogether, in view of \cite[Theorem~2.4.9]{Zalinescu},
$g|_Y$ is continuous. Hence $g$ is lower semicontinuous.
Using \cite[Corollary~2.4.15]{Zalinescu} and
Fact~\ref{Rea:1}\ref{Sia:2b}, we now deduce that
$(\forall y\in Y)$
$g(y) = \sup\scal{\partial f(x)}{y} = \scal{Ax}{y} = \scal{x}{A^*y}$.
We thus have verified that
\begin{equation}
\label{e:091208:a}
(\forall x\in Y)(\forall y\in Y)\quad
f'(x;y) =  \scal{Ax}{y} = \scal{x}{A^*y}.
\end{equation}
In particular, $f|_Y$ is \emph{differentiable}.
Now fix $x,y,z$ in $Y$.
Then, using \eqref{e:091208:a}, we see that
\begin{align}
\label{e:091208:b}
\scal{Az}{y} &=
\lim_{s\downarrow 0} \frac{\scal{A(x+sz)}{y} -
\scal{Ax}{y}}{s} =
\lim_{s\downarrow 0} \frac{f'(x+sz;y)-f'(x;y)}{s}\\
&= \lim_{s\downarrow 0} \lim_{t\downarrow 0}
\Big(\frac{f(x+sz+ty)-f(x+sz)}{st} - \frac{f(x+ty)-f(x)}{st}\Big).\notag
\end{align}
Set $h\colon\RR\to\RR\colon s\mapsto f(x+sz+ty)-f(x+sz)$.
Since $f|_Y$ is differentiable, so is $h$.
For $s>0$, the Mean Value Theorem thus yields
$r_{s,t}\in \left]0,s\right[$
such that
\begin{align}
\label{e:091208:c}
\frac{f(x+sz+ty)-f(x+sz)}{s} - \frac{f(x+ty)-f(x)}{s}
&= \frac{h(s)}{s}-\frac{h(0)}{s} = h'(r_{s,t})\\
&=
f'(x+r_{s,t}z+ty;z) - f'(x+r_{s,t}z;z)\notag\\
&=t\scal{Ay}{z}. \notag
\end{align}
Combining \eqref{e:091208:b} with \eqref{e:091208:c},
we deduce that $\scal{Az}{y}=\scal{Ay}{z}$.
Thus, $A$ is symmetric.
The result now follows from Proposition~\ref{symm=adj}.
\end{proof}

To improve Lemma~\ref{Bor:u1}, we need the following ``shrink and dilate''
technique.

\begin{lemma}\label{dilation}
Let $A:X\To X^*$ be a monotone linear relation, and
let $Z$ be a closed subspace of $\dom A$.
Set $B = (A+\II_Z)+Z^\bot$. 
Then $B$ is maximal monotone and $\dom B=Z$.
\end{lemma}
\begin{proof}
Since $Z\subseteq \dom A$ and $B = A+\partial\iota_Z$ 
it is clear that 
$B$ is a monotone linear relation with $\dom B = Z$.
By Fact~\ref{linear}~\ref{Nov:s1}, we have
\begin{equation*}
Z^{\perp}\subseteq B0 = A0+Z^{\perp}
\subseteq(\dom A)^{\perp}+Z^{\perp}
\subseteq Z^{\perp}+Z^{\perp}=Z^{\perp}.
\end{equation*}
Hence $B0 = Z^\bot = (\dom B)^\bot$. 
Therefore, by Fact~\ref{linear}\ref{sia:3v},
$B$ is maximal monotone.
\end{proof}

\begin{theorem}\label{Bor:U2}
Let $f:X\to\RX$ be  proper,
lower semicontinuous, and convex, and
let $Y$ be a linear subspace of $X$.
Suppose that $\partial f + \II_Y$ is a linear relation.
Then $\partial f + \II_Y$ is symmetric.
\end{theorem}
\begin{proof}
Put $A=\partial f+\II_{Y}$.
Assume that
$(x,x^*), (y,y^*) \in \gra A$.
Set $Z=\spand\{x,y\}$.
Let $B:X\To X^*$ be defined as in Lemma~\ref{dilation}.
Clearly, $\gra B\subseteq \gra \partial (f+\iota_{Z})$.
In view of the maximal monotonicity of $B$, we see that
$B=\partial(f+\iota_Z)$.
Since $\dom B=Z$ is closed, it follows from
Lemma~\ref{Bor:u1} that $B=B^*$.
In particular, we obtain that $\scal{x^*}{y}=\scal{y^*}{x}$.
Hence, $\scal{\partial f(x)}{y}=\scal{\partial f(y)}{x}$
and therefore $\partial f + \II_Y$ is symmetric.
\end{proof}

\begin{lemma}\label{theo:5}
Let $A\colon X \To X^*$
be a maximal monotone linear relation such that
$A$ is Borwein-Wiersma decomposable.
Then $\dom A\subseteq\dom A^*$.
\end{lemma}
\begin{proof}
By hypothesis,
there exists a proper lower semicontinuous
and convex function $f\colon X\to\RX$ and an
at most single-valued skew operator $S$
such that $A=\partial f+S$.
Hence $\dom A\subseteq\dom S$, and
Theorem~\ref{Bor:U2} implies that $(A-S) + \II_{\dom A}$ is symmetric.
Let $x$ and $y$ be in $\dom A$.
\begin{align*}
\langle Ax-2Sx,y\rangle &=\langle Ax-Sx,y\rangle-\langle Sx,y\rangle
=\langle Ay-Sy,x\rangle-\langle Sx,y\rangle\\
&=\langle Ay,x\rangle-\langle Sy,x\rangle-\langle Sx,y\rangle
=\langle Ay,x\rangle,
\end{align*}
which implies that $(A-2S)x\subseteq A^*x$.
Therefore, $\dom A=\dom (A-2S)\subseteq\dom A^*$.
\end{proof}

\begin{remark}
We can now derive part of the conclusion of
of Proposition~\ref{Co:r2} differently as follows.
Since $\dom A_1 - \dom A_2$ is closed,
\cite[Theorem~5.5]{SiZ} or \cite{Voi} implies that
$A_1+A_2$ is maximal monotone; moreover,
\cite[Theorem~7.4]{Borwein} yields
$(A_1+A_2)^* = A_1^* + A_2^*$.
Using Lemma~\ref{theo:5}, we thus obtain
$\dom(A_1+A_2) = \dom A_1 \cap \dom A_2
\subseteq \dom A_1^* \cap \dom A_2^* =
\dom(A_1^*+A_2^*) = \dom (A_1+A_2)^*$.
Therefore, $A_1+A_2$ is Borwein-Wiersma decomposable by
Theorem~\ref{theo:4}.
\end{remark}

\begin{theorem}[characterization of subdifferential operators]
\label{Bor:U3}
Let $A:X\To X^*$ be a monotone linear relation.
Then $A$ is maximal monotone and symmetric $\Leftrightarrow$
there exists a proper lower semicontinuous convex
function $f\colon X\to\RX$ such that $A=\partial f$.
\end{theorem}
\begin{proof}
``$\Rightarrow$'': {Fact}~\ref{f:PheSim}\ref{f:PheSim:quad+}.
``$\Leftarrow$'': Apply Theorem~\ref{Bor:U2} with $Y=X$.
\end{proof}

\begin{remark}
Theorem~\ref{Bor:U3} generalizes
\cite[Theorem~5.1]{PheSim} of Simons and Phelps.
\end{remark}

\begin{theorem}[uniqueness of the subdifferential part]\label{uniquepart}
Let $A:X\To X^*$ be a maximal monotone linear relation
such that $A$ is Borwein-Wiersma decomposable.
Then on $\dom A$,
the subdifferential part in the decomposition is unique
and the skew part must be a linear selection of
$A_{\mathlarger{\circ}}$.
\end{theorem}
\begin{proof}
Let $f_1$ and $f_2$ be proper lower semicontinuous convex functions
from $X$ to $\RX$, and let $S_1$ and $S_2$ be at most single-valued
skew operators from $X$ to $X^*$ such that
\begin{equation}
\label{Borw:7}
A=\partial f_1+S_1=\partial f_2+S_2.
\end{equation}
Set $D = \dom A$.
Since $S_1$ and $S_2$ are single-valued on $D$, we have
$A-S_1=\partial f_1$ and  $A-S_2=\partial f_2$ on $D$.
Hence $\partial f_1 + \II_{D}$ and $\partial f_2+\II_{D}$ are
monotone linear relations with
\begin{equation}
\label{e:091208n:a}
(\partial f_1 + \II_{D})(0) = (\partial f_2 + \II_{D})(0)=A0.
\end{equation}
By Theorem~\ref{Bor:U2},
$\partial f_1 + \II_{D}$ and $\partial f_2+\II_{D}$ are
symmetric, i.e.,
$$
(\forall x\in D)(\forall y\in D)\quad
\langle \partial f_1 (x),y\rangle=\langle \partial f_1(y),x\rangle
\quad\text{and}\quad
\langle \partial f_2 (x),y\rangle=\langle \partial f_2(y),x\rangle .$$
Thus,
\begin{equation}
\label{Bor:u9}
(\forall x\in D)(\forall y\in D)\quad
\langle \partial f_2 (x)-\partial f_1 (x),y\rangle
=\langle \partial f_2 (y)-\partial f_1 (y),x\rangle.
\end{equation}
On the other hand, by \eqref{Borw:7},
$(\forall x\in D)$ $S_1x-S_2x\in\partial f_2(x)-\partial f_1(x)$.
Then by Fact~\ref{linear}\ref{Sia:2c} and Proposition~\ref{MPP},
\begin{align}
\label{Bor:u10}
(\forall x\in D)(\forall y\in D)\quad
\langle \partial f_2(x)-\partial f_1(x),y\rangle &=
\langle S_1x-S_2x,y\rangle\\
&=-\langle S_1y-S_2y,x\rangle\notag\\
&=-\langle \partial f_2 (y)-\partial f_1 (y),x\rangle. \notag
\end{align}
Now fix $x\in D$.
Combining \eqref{Bor:u9} and \eqref{Bor:u10}, we get
$(\forall y\in D)$
$\langle \partial f_2(x)-\partial f_1(x),y\rangle=0$.
Using {Fact}~\ref{linear}\ref{sia:3iv}, we see that
$$\partial f_2(x)-\partial f_1(x)\subseteq D^\bot = (\dom A)^{\bot}=A0.$$
Hence, in view of {Lemma}~\ref{amy}\ref{domlinear}, \eqref{e:091208n:a},
and Fact~\ref{Rea:1}\ref{Th:28},
$$ \partial f_1 + \II_D =\partial f_2 + \II_D.$$
Furthermore, combining \eqref{Bor:u9} and \eqref{Bor:u10} gives
$(\forall y\in D)$
$\langle  S_2x-S_1x,y\rangle=0$; thus,
$$ S_2x-S_1x\in D^\bot = (\dom A)^{\bot}=A0.$$
Now Lemma~\ref{theo:5} implies that
$\dom A\subseteq \dom A^*$.
In turn, Theorem~\ref{theo:4} allows us to consider the case when
$S_1$ is a linear selection of $A_{\mathlarger{\circ}}$ on
$\dom A_{\mathlarger{\circ}} = \dom A$.
Using Fact~\ref{linear}\ref{sia:3iv}, we obtain
Then $S_2 x\in S_1x+A0=S_1x+A_{\mathlarger{\circ}}0
=A_{\mathlarger{\circ}}x$.
Therefore, $S_2$ must be a linear selection of
$A_{\mathlarger{\circ}}$ on $\dom A$ as well.
\end{proof}

\begin{remark} In a Borwein-Wiersma decomposition, 
the skew part need not be unique:
indeed, assume that $X=\RR^2$, set 
$Y:=\RR\times\{0\}$, and let $S$ be given by
$\gra S =\menge{\big((x,0), (0,x)\big)}{x\in\RR}$. 
Then $S$ is skew and the maximal monotone linear relation $\partial
\iota_Y$ has two distinct Borwein-Wiersma decompositions, namely
$\partial \iota_{Y}+0$ and $\partial\iota_{Y}+S$.
\end{remark}

\begin{proposition}
Let $A:X\To X^*$ be a maximal monotone linear relation.
Suppose that $A$ is Borwein-Wiersma decomposable, with subdifferential part
$\partial f$, where
$f\colon X\to\RX$ is proper, lower semicontinuous and convex.
Then there exists a constant $\alpha\in\RR$ such that the following hold.
\begin{enumerate}
\item
\label{w:eins}
$f=\overline{q_A}+\alpha$ on $\dom A$.
\item
\label{w:zwei}
If $\dom A$ is closed, then
$f=\overline{q_A}+\alpha = q_A + \alpha$ on $X$.
\end{enumerate}
\end{proposition}
\begin{proof}
Let $S$ be a linear single-valued selection of $A_{\mathlarger{\circ}}$.
By Lemma~\ref{theo:5}, $\dom A\subseteq\dom A^*$.
In turn, Theorem~\ref{theo:4} yields
$$A=\partial \overline{q_A} +S.$$
Let 
$\{x,y\}\subset \dom A$. 
By Theorem~\ref{uniquepart},
$\partial f + \II_{\dom A} = \partial\overline{q_A}+\II_{\dom A}$.
Now set $Z = \spand\{x,y\}$, 
apply Lemma~\ref{dilation} to the
monotone linear relation $\partial f + \II_{\dom A} =
\partial\overline{q_A}+\II_{\dom A}$, and let $B$ be as in
Lemma~\ref{dilation}. 
Note that
$\gra B =  \gra(\partial\overline{q_A} + \partial\iota_{Z}) \subseteq
\gra\partial(\overline{q_A}+\iota_Z)$
and that $\gra B = \gra(\partial f + \partial\iota_{Z})\subseteq
\gra\partial(f+\iota_Z)$.
By maximal monotonicity of $B$, we conclude that
$B= \partial(\overline{q_A}+\iota_Z) = \partial(f+\iota_Z)$.
By \cite[Theorem~B]{Rocksub}, there exists $\alpha\in\RR$ such that
$f+\iota_{Z}=\overline{q_A}+\iota_{Z}+\alpha$.
Hence $\alpha = f(x)-\overline{q_A}(x)=f(y)-\overline{q_{A}}(y)$ and repeating this argument
with $y\in(\dom A)\smallsetminus\{x\} $, we see that
\begin{equation}\label{differenceconst}
f=\overline{q_{A}}+\alpha \quad\text{on $\dom A$}
\end{equation}
and \ref{w:eins} is thus verified.
Now assume in addition that $\dom A$ is closed.
Applying Lemma~\ref{amy}\ref{domainpf++} with both
$\partial f$ and $\partial\overline{q_A}$, we obtain
$$\dom \overline{q_A}=\dom \partial \overline{q_A}=
\dom A=\dom \partial f=\dom f.$$
Consequently, \eqref{differenceconst} now yields
$f=\overline{q_A} +\alpha$. Finally, Corollary~\ref{cbr:1}
implies that $q_A = \overline{q_A}$.
\end{proof}

\section{Characterizations and examples}
\label{characterization}

The following characterization of Borwein-Wiersma
decomposability of a maximal monotone linear relation is quite pleasing.

\begin{theorem}[characterization of Borwein-Wiersma decomposability]\label{theo:9}
Let $A\colon X \rightrightarrows X^*$
be a maximal monotone linear relation.
 Then the following are equivalent.
\begin{enumerate}
\item
\label{theo:9i}
$A$ is Borwein-Wiersma decomposable.
\item
\label{theo:9ii}
$\dom A\subseteq\dom A^*$.
\item
\label{theo:9iii}
$A=A_++A_{\mathlarger{\circ}}$.
\end{enumerate}
\end{theorem}
\begin{proof}
``\ref{theo:9i}$\Rightarrow$\ref{theo:9ii}'':
 Lemma~\ref{theo:5}.
``\ref{theo:9i}$\Leftarrow$\ref{theo:9ii}'':
Theorem~\ref{theo:4}.
``\ref{theo:9ii}$\Rightarrow$\ref{theo:9iii}'':
{Fact}~\ref{linear}\ref{sia:3vii}.
``\ref{theo:9ii}$\Leftarrow$\ref{theo:9iii}'':
This is clear. 
\end{proof}

\begin{corollary}\label{Co:c1}
Let $A\colon X \To X^*$ be a maximal monotone linear relation.
Then both $A$ and $A^*$ are Borwein-Wiersma decomposable
if and only if $\dom A=\dom A^*$.
\end{corollary}
\begin{proof}
Combine Theorem~\ref{theo:9}, Fact~\ref{Sv:7}, and
Fact~\ref{Rea:1}\ref{Th:31}.
\end{proof}

We shall now provide two examples of a linear relation $S$
to illustrate that the following do occur:
\begin{itemize}
\item $S$ is  Borwein-Wiersma decomposable, but $S^*$ is not.
\item Neither $S$ nor $S^*$ is Borwein-Wiersma decomposable.
\item $S$ is not Borwein-Wiersma decomposable, but $S^{-1}$ is.
\end{itemize}

\begin{example}
(See \cite{BWY7}.)
\label{FE:1}
Suppose that $X$ is the Hilbert space $\ell^2$, and set
\begin{align}
\label{EL:1}
S\colon \dom S\to X\colon
y\mapsto
\bigg(\thalb y_n + \sum_{i<n}y_{i}\bigg),
\end{align}
with
\begin{equation*}
\dom S:=\bigg\{ y=(y_n)\in X \;\;\bigg|\;\; \sum_{i\geq 1}y_{i}=0,
\bigg(\sum_{i\leq n}y_{i}\bigg)\in X\bigg\}.
\end{equation*}
Then
\begin{align}
\label{PF:a2}
S^*\colon \dom S^* \to X\colon
y\mapsto \bigg(\thalb y_n + \sum_{i> n}y_{i}\bigg)
\end{align}
where
\begin{equation*}
\dom S^*=\bigg\{ y=(y_n)\in X\;\; \bigg|\;\;
 \bigg(\sum_{i> n}y_{i}\bigg)\in X\bigg\}.
\end{equation*}
Then $S$ can be identified with an at most single-valued linear relation
such that the following hold.
(See \cite[Theorem~2.5]{PheSim} and \cite[Proposition~3.2, Proposition~3.5,
Proposition~3.6, and Theorem~3.9]{BWY7}.)
\begin{enumerate}
 \item
 $S$ is maximal monotone and skew.
\item $S^*$ is maximal monotone but not skew.
\item $\dom S$ is dense in $\ell^2$, and  $\dom S\subsetneqq\dom S^*$.
\item $S^* = -S$ on $\dom S$.
\end{enumerate}
In view of Theorem~\ref{theo:9},
$S$ is Borwein-Wiersma decomposable while $S^*$ is not.
However, both  $S$ and  $S^*$ are irreducible and Asplund decomposable
by Theorem~\ref{acyclicmono}.
Because $S^*$ is irreducible but not skew,
we see that the class of irreducible operators is strictly larger
than the class of skew operators.
\end{example}

\begin{example}[inverse Volterra operator]
\label{FE:2}
(See \cite[Example~4.4 and Theorem~4.5]{BWY7}.)
Suppose that $X$ is the Hilbert space $L^2[0,1]$, and consider
the \emph{Volterra integration operator}
(see, e.g., \cite[Problem~148]{Halmos}), which is defined by
\begin{equation}
\label{e:aug17:a}
V \colon X \to X \colon x \mapsto Vx, \quad\text{where}\quad
Vx\colon[0,1]\to\RR\colon t\mapsto \int_{0}^{t}x,
\end{equation}
and set $A = V^{-1}$.
Then
\begin{equation*}
V^*\colon X\to X\colon x\mapsto V^*x, \quad\text{where}\quad
V^*x\colon[0,1]\to\RR\colon t\mapsto \int_{t}^{1}x,
\end{equation*}
and the following hold.
\begin{enumerate}
\item
$\dom A=\big\{x\in X\;\big|\;\text{$x$ is absolutely continuous,
$x(0)=0$, and $x'\in X$}\big\}$ and
$$A\colon\dom A\to X\colon x\mapsto x'.$$
\item
$\dom A^*=\big\{x\in X\;\big|\;\text{$x$ is absolutely continuous,
$x(1)=0$, and $x'\in X$}\big\}$ and
$$A^*\colon\dom A^*\to X\colon x\mapsto -x'.$$
\item Both $A$ and $A^*$ are maximal monotone linear operators.
\item Neither $A$ nor $A^*$ is symmetric.
\item Neither $A$ nor $A^*$ is skew.
\item $\dom A \not\subseteq \dom A^*$, and
$\dom A^* \not\subseteq \dom A$.
\item $Y := \dom A \cap \dom A^*$ is dense in $X$.
\item Both $A+\II_Y$ and $A^*+\II_Y$ are skew.
\end{enumerate}
By {Theorem}~\ref{acyclicmono}, both
$A$ and $A^*$ are irreducible and Asplund decomposable.
On the other hand, by Theorem~\ref{theo:9},
neither $A$ nor $A^*$ is Borwein-Wiersma decomposable.
Finally, $A^{-1}=V$ and $(A^*)^{-1}=V^*$ are Borwein-Wiersma decomposable
since they are continuous linear operators with full domain.
\end{example}

\begin{remark}[an answer to a question posed by Borwein and Wiersma]
The operators $S$, $S^*$, $A$, and $A^*$ defined in this section are all
irreducible and Asplund decomposable, but none of them has full domain.
This provides an answer to \cite[Question~(4) in
Section~7]{BorweinWiersma}.
\end{remark}

\section{When $X$ is a Hilbert space}
\label{explicit}

Throughout this short section, we suppose that $X$ is a Hilbert space.
Recall (see, e.g., \cite[Chapter~5]{deutsch} for basic properties)
that if $C$ is a nonempty closed convex subset of $X$, then the
\emph{(nearest point) projector} $P_C$ is well defined and continuous.
If $Y$ is a closed subspace of $X$, then $P_Y$ is linear and
$P_Y = P_Y^*$.

\begin{definition}
Let $A\colon X \To X$ be a maximal monotone linear relation.  We define
$Q_A$ by
\begin{equation*}
Q_A:\dom A\rightarrow X:x\mapsto P_{Ax}x.
\end{equation*}
\end{definition}
Note that $Q_A$ is monotone and a single-valued selection of $A$ because
$(\forall x\in \dom A)$ $Ax$ is a
nonempty closed convex subset of $X$.

\begin{proposition}[linear selection]
\label{pa2}
Let $A\colon X \To X$ be a maximal monotone linear relation.
Then the following hold.
\begin{enumerate}
\item
\label{t:2}
$(\forall x\in\dom A)$
$Q_{A}x=P_{(A0)^\bot}(Ax)$, and $Q_{A}x\in Ax$.
\item
\label{t:2+}
$Q_A$ is monotone and linear.
\item
\label{t:2++}
$A=Q_{A}+A0.$
\end{enumerate}
\end{proposition}
\begin{proof}
Let $x\in\dom A = \dom Q_A$ and let $x^*\in Ax$.
Using Fact~\ref{Rea:1}\ref{Th:28} and {Fact}~\ref{linear}\ref{Nov:s1},
we see that
\begin{align*}
Q_{A}x&=P_{Ax}x=P_{x^*+A0}x = x^*+P_{A0}(x-x^*)=x^*+P_{A0}x-P_{A0}x^*
=P_{A0}x+P_{(A0)^\bot}x^*\\
&=P_{(A0)^\bot}x^*.
\end{align*}
Since $x^*\in Ax$ is arbitrary, we have thus verified \ref{t:2}.
Now let $x$ and $y$ be in $\dom A$, and let $\alpha$ and $\beta$ be in
$\RR$.
If $\alpha=\beta=0$, then, by Fact~\ref{Rea:1}\ref{Th:26},
we have
$Q_A (\alpha x+\beta y)=Q_A0=P_{A0}0=0=\alpha Q_A x+\beta Q_A y$.
Now assume that $\alpha\neq 0$ or $\beta\neq 0$.
By \ref{t:2} and Fact~\ref{Rea:1}\ref{Th:30}, we have
 \begin{equation*}
 Q_A(\alpha x+\beta y)=P_{(A0)^\bot}A(\alpha x+\beta y)=
\alpha P_{(A0)^\bot}(Ax)+\beta P_{(A0)^\bot}
(Ay) =\alpha Q_A x+\beta Q_A y.
\end{equation*}
Hence $Q_A$ is a linear selection of $A$ and \ref{t:2+} holds.
Finally, \ref{t:2++} follows from Fact~\ref{Rea:1}\ref{Th:28}.
\end{proof}

\begin{example}
\label{theo:40}
Let $A\colon X \To X$ be maximal monotone and skew.
Then $A=\partial\iota_{\overline{\dom A}}+Q_A$ is a
Borwein-Wiersma decomposition.
\end{example}
\begin{proof}
By Proposition~\ref{pa2}\ref{t:2+}, $Q_A$ is
a linear selection of $A$. Now apply Theorem~\ref{Sk:1}.
\end{proof}

\begin{example}\label{EBor:1}
Let $A\colon X \To X$ be a maximal monotone linear relation such
that $\dom A$ is closed.
Set $B := P_{\dom A}Q_A P_{\dom A}$ and $f := q_B+\iota_{\dom A}$.
Then the following hold.
\begin{enumerate}
\item
\label{EBor:1i}
$B\colon X\to X$ is continuous, linear, and maximal monotone.
\item
\label{EBor:1ii}
$f\colon X\to\RX$ is convex, lower semicontinuous, and proper.
\item
\label{EBor:1iii}
$A = \partial\iota_{\dom A} + B$.
\item
\label{EBor:1iv}
$\partial f + B_{\mathlarger{\circ}}$
is a Borwein-Wiersma decomposition of $A$.
\end{enumerate}
\end{example}
\begin{proof}
\ref{EBor:1i}:
By Proposition~\ref{pa2}\ref{t:2+},
$Q_A$ is monotone and  a linear selection of $A$. Hence,
$B\colon X\to X$ is linear; moreover,
$(\forall x\in X)$
$\langle x,B
x\rangle=\langle x,\ P_{\dom A}Q_A P_{\dom A}x\rangle
=\langle P_{\dom A}x,\ Q_A P_{\dom A}x\rangle\geq 0$.
Altogether, $B\colon X\to X$ is linear and monotone.
By Fact~\ref{F:1}, $B$ is continuous and maximal monotone.

\ref{EBor:1ii}:
By \ref{EBor:1i}, $q_B$ is thus convex and continuous;
in turn, $f$ is convex, lower semicontinuous, and proper.

\ref{EBor:1iii}:
Using Proposition~\ref{pa2}\ref{t:2} and {Fact}~\ref{linear}\ref{sia:3iv},
we have $(\forall x\in X)$
$(Q_AP_{\dom A}) x\in (A0)^\bot=\overline{\dom A}=\dom A$.
Hence,
$(\forall x\in \dom A)$
$Bx=(P_{\dom A}Q_AP_{\dom A})x=Q_Ax\in Ax$.
Thus, $B+\II_{\dom A} = Q_A$.
In view of Proposition~\ref{pa2}\ref{t:2++}
and {Fact}~\ref{linear}\ref{sia:3iv},
we now obtain
$A = B+\II_{\dom A} + A0 = B+\partial\iota_{\dom A}$.

\ref{EBor:1iv}:
It follows from \ref{EBor:1iii} and \eqref{capitalnews} that
$A = B+ \partial\iota_{\dom A} = \nabla q_{B}+\partial \iota_{\dom A}+B_{\mathlarger{\circ}}
 =\partial
(q_{B}+\iota_{\dom A})+B_{\mathlarger{\circ}}
=\partial
f+B_{\mathlarger{\circ}}$.
\end{proof}


\begin{proposition}
Let $A:X\To X$ be such that $\dom A$ is a closed subspace of $X$.
Then $A$ is a maximal monotone linear relation
$\Leftrightarrow$
$A=\partial\iota_{\dom A} + B$, where
$B:X\to X$ is linear and monotone.
\end{proposition}
\begin{proof}
``$\Rightarrow$'': This is clear from
Example~\ref{EBor:1}\ref{EBor:1i}\&\ref{EBor:1iii}.
``$\Leftarrow$'': Clearly, $A$ is a linear relation.
By Fact~\ref{F:1}, $B$ is continuous and maximal
monotone. Using Rockafellar's sum theorem \cite{Rockf}, we conclude
that $\partial\iota_{\dom A} + B$ is maximal monotone.
\end{proof}


\section{Conclusion}\label{summary}
The original papers by Asplund \cite{Asplund} and by Borwein and Wiersma
\cite{BorweinWiersma} concerned the additive decomposition of a maximal
monotone operator whose domain has \emph{nonempty interior}.
In this paper, we focused on maximal monotone \emph{linear relations} and
we specifically allowed for domains with empty interior.
All maximal monotone linear relations on finite-dimensional spaces
are Borwein-Wiersma decomposable; however, this fails
in infinite-dimensional settings.
We presented characterizations of Borwein-Wiersma
decomposability of maximal monotone linear relations
in reflexive Banach spaces and provided a more explicit decomposition
in Hilbert spaces.

The characterization of Asplund decomposability and the corresponding
construction of an Asplund decomposition
remain interesting unresolved topics for future explorations,
even for maximal monotone linear operators whose domains are proper
dense subspaces of infinite-dimensional Hilbert spaces.

\section*{Acknowledgments}
Heinz Bauschke was partially supported by the Natural Sciences and
Engineering Research Council of Canada and
by the Canada Research Chair Program.
Xianfu Wang was partially supported by the Natural
Sciences and Engineering Research Council of Canada.



\end{document}